\documentclass[reqno]{amsart}

\usepackage{amsmath}
\usepackage{amsthm}
\usepackage{amsfonts}
\usepackage{amssymb}
\usepackage{enumerate}
\usepackage{graphicx}

\DeclareMathOperator{\tr}{tr}

\DeclareMathOperator{\rank}{rank}

\newcommand{\Prob}{\mathbb{P}}

\def\R{\mathbb{R}}

\theoremstyle{plain}
  \newtheorem{theorem}{Theorem}
  \newtheorem{conjecture}{Conjecture}

  \newtheorem{lemma}{Lemma}
  \newtheorem{corollary}{Corollary}
  
  \newtheorem{question}{Question}

\theoremstyle{definition}
  \newtheorem{definition}{Definition}
  
  \newtheorem{remark}{Remark}

\begin{document}
\title[Controllability of random systems]{Controllability of random systems: Universality and minimal controllability} 

\author{Sean O'Rourke}
\address{Department of Mathematics, University of Colorado at Boulder, Boulder, CO 80309 }
\email{sean.d.orourke@colorado.edu}

\author{Behrouz Touri}
\address{Department of Electrical, Computer, and Energy Engineering, University of Colorado at Boulder, Boulder, CO 80309}
\email{behrouz.touri@colorado.edu}

\begin{abstract}
For a large class of random matrices $A$ and vectors $b$, we show that linear systems formed from the pair $(A,b)$ are controllable with high probability. Despite the fact that minimal controllability problems are, in general, NP-hard, we establish universality results for the minimal controllability of random systems.  Our proof relies on the recent developments of Nguyen--Tao--Vu \cite{NTV} concerning gaps between eigenvalues of random matrices.  
\end{abstract}

\maketitle

\section{Introduction}
In recent years, the topic of controllability of linear systems has regained a great deal of attention due to its potential impact on the controllability of networked systems. Several researchers have provided conditions on controllability of specific linear systems with a given set of parameters. These works include \cite{bahman,chapman,godsil,barabasi,marzieh,magnus} where several necessary and sufficient conditions have been obtained for controllability of a system with a given transition matrix and a given input vector/matrix.  Also, it was shown in \cite{alex} that the problem of minimal controllability of a linear system (i.e.\ finding the sparsest input vector that makes a given linear system controllable) is NP-hard, in general. 

Motivated by these works, we study the controllability of random systems.  That is, we consider linear systems whose parameters (as far as controllability is concerned) are random.  We confirm a common belief that ``most systems are highly controllable'' even when one deals with systems of a very discrete nature.  

For the purposes of this note, we define \textit{controllability} in terms of Kalman's rank condition \cite{kalman1962controllability}. 
\begin{definition}[Controllable]
Let $A$ be an $n \times n$ matrix, and let $b$ be a vector in $\R^n$.  We say the pair $(A,b)$ is \emph{controllable} if the $n \times n$ matrix
\begin{equation} \label{eq:Ab}
	\begin{pmatrix} b & Ab & \cdots & A^{n-1}b \end{pmatrix}
\end{equation}
has full rank (that is, rank $n$).  Here, the matrix in \eqref{eq:Ab} is the matrix with columns $b, Ab, \ldots, A^{n-1} b$.
\end{definition}

Our work is an attempt to address the following commonly believed conjecture.  
\begin{conjecture}\label{conj:conjecture1}
As $n\to\infty$, the probability that $(A,e_i)$ is controllable for all $i\in \{1,\ldots,n\}$ approaches one, where $A$ is the adjacency matrix of an Erd\H{o}s--R\'{e}nyi random graph $G(n,p)$, and $e_i$ is the $i$th member of the standard basis in $\R^n$.
\end{conjecture}
In addition, we also discuss the following conjecture from \cite{godsil}\footnote{The conjecture in \cite{godsil} is slightly different than the one stated here as it involves the uniform probability distribution on graphs with $n$ vertices.}.
\begin{conjecture}[\cite{godsil}]\label{conj:conjecture2}
As $n\to\infty$, the probability that $(A,1_n)$ is controllable approaches one, where $A$ is the adjacency matrix of an Erd\H{o}s--R\'{e}nyi random graph $G(n,p)$ and $1_n$ is the all-one vector in $\R^n$.
\end{conjecture}

We prove a stronger version of the first conjecture and provide some insights into the second conjecture. Our work heavily relies on the recent advancements in inverse Littlewood-Offord theory (see \cite{NV} for a survey) and its implications on eigenvalue spacing of random matrices recently studied by Nguyen, Tao, and Vu \cite{NTV}. We prove that, under a very general assumption on the random matrix $A$ and the (deterministic or random) vector $b$, the pair $(A,b)$ is controllable with high probability. We also establish that, for a general class of random matrices $A$, the pair $(A,e_i)$ is controllable \textit{for all} $i\in \{1,\ldots,n\}$ with high probability.  Our results suggest that, perhaps, the problem of controllability and minimal controllability is not a concerning problem for large systems.  

\subsection{Outline and notation}
The paper is organized as follows.  We introduce the main objects of study in Section~\ref{sec:prelim}, and we state our main results in Section~\ref{sec:mainresults}. We then present the technical tools we will use to prove our main results in Section~\ref{sec:tools}.  The proofs of our main results are presented in Section~\ref{sec:proof}.  Finally, Section \ref{sec:closing} contains a few closing remarks and some related open problems.

Throughout the paper, we use the following notation.  For any matrix $A$, $\|A\|$ is the spectral norm of $A$.  A matrix obtained by deleting the $i$th row and the $i$th column of $A$ is called a first minor of $A$. Since we only work with first minors, we simply refer to this matrix as the $i$th minor of $A$.  We let $A^\mathrm{T}$ be the transpose of $A$, and $A^\ast$ is the conjugate transpose of $A$. 

For an $n \times n$ Hermitian matrix $A$, we let 
$$ \lambda_1(A) \leq \cdots \leq \lambda_n(A) $$ denote the (ordered) eigenvalues of $A$ (counting algebraic multiplicity) with corresponding orthonormal eigenvectors $v_1(A), \ldots, v_n(A)$.  Since the eigenvectors of $A$ are not uniquely determined, unless otherwise stated, we always assume that the collection $\{v_1(A), \ldots, v_n(A)\}$ is any orthonormal basis of eigenvectors of $A$ such that 
$$ A v_i(A) = \lambda_i(A) v_i(A). $$  
We emphasis the fact that, throughout this paper, all the eigenvectors under consideration are unit length.  

We let $e_1, \ldots, e_n$ denote the $n$ elements of the standard basis in $\mathbb{R}^n$.  If $u$ and $v$ are vectors in $\mathbb{C}^n$ (alternatively $\mathbb{R}^n$), then $\|u\|$ is the Euclidean norm of $u$, and $u^\ast v$ (respectively $u^\mathrm{T} v$) denotes the standard Euclidean inner product of $u$ and $v$. 

For an event $E$, we let $E^c$ denote the complement of $E$. For a random variable $\xi$ with a given distribution, if there is no confusion, we refer to the distribution of $\xi$ simply as $\xi$.

Throughout this paper, we consider controllability of single-input systems. We often refer to a linear system whose transition matrix $A$ or input vector $b$ are random, as a \textit{random system}.

\section{Random matrix ensembles}\label{sec:prelim}
The focus of this paper is on controllability of a random matrix $A$ and a (deterministic or random) vector $b$.  In particular, we address the following question.  
\begin{question} \label{quest:main}
Let $b \in \R^n$ be nonzero.  If $A$ is an $n \times n$ random Hermitian matrix, what is the probability that $(A, b)$ is controllable?
\end{question}

Of course, the answer to Question \ref{quest:main} depends on the distribution of $A$ and the particular choice of $b$.  In this note, we consider two important ensembles of real symmetric random matrices.  

\subsection{Wigner random matrices}
Wigner \cite{W} introduced the following class of real symmetric random matrices with independent entries in the 1950s.

\begin{definition}[Wigner matrix]
Let $\xi$ and $\zeta$ be real random variables.  We say $W = (w_{ij})_{i,j=1}^n$ is an $n \times n$ \emph{Wigner matrix} with atom variables $\xi$ and $\zeta$ if $W$ is a real symmetric matrix whose entries satisfy the following:
\begin{itemize}
\item the entries $\{w_{ij} : 1 \leq i \leq j \leq n\}$ are independent random variables,
\item the upper triangular entries $\{w_{ij} : 1 \leq i < j \leq n\}$ are iid (independent and identically distributed) copies of $\xi$,
\item the diagonal entries $\{w_{ii} : 1 \leq i \leq n\}$ are iid copies of $\zeta$.
\end{itemize}
\end{definition}

A classical example of  a Wigner matrix is the \emph{adjacency matrix of a random graph}.  Let $G(n,p)$ be the Erd\H{o}s--R\'{e}nyi random graph on $n$ vertices with edge density $p$, i.e.\ it is a graph on $n$ vertices such that each pair of vertices $\{i,j\}$ is an edge in $G(n,p)$ with probability $p$, independent of other edges.  We denote by $A = (a_{ij})_{i,j=1}^n$ the zero-one adjacency matrix of $G(n,p)$.  That is,
$$ a_{ij} := \left\{
     \begin{array}{lr}
       1, & \text{if } i \sim j, \\
       0, & \text{if } i \not\sim j.
     \end{array}
   \right. $$
In particular, the upper-diagonal entries of $A$ are iid copies of a zero-one Bernoulli random variable (with parameter $p$), and the diagonal entries are all zero.

In general, we will be interested in Wigner random matrices whose entries are sub-gaussian.

\begin{definition}[Sub-gaussian]
We say the random variable $\xi$ is \emph{sub-gaussian} with sub-gaussian moment $\kappa > 0$ if
$$ \Prob( |\xi| \geq t) \leq \kappa^{-1} \exp(- \kappa t^2) $$
for all $ t > 0$.
\end{definition}

There are many examples of sub-gaussian random variables.  For instance, Gaussian random variables are sub-gaussian.  In addition, all bounded random variables (i.e.\ random variables that are supported on a bounded subset of $\R$) are sub-gaussian.

\subsection{Orthogonally invariant ensembles}
We also consider a class of orthogonally invariant ensembles.

\begin{definition}[Orthogonally invariant ensemble]
We say the $n \times n$ random symmetric matrix $M$ is drawn from an \emph{orthogonally invariant ensemble} if $M$ has probability distribution
\begin{equation} \label{eq:oe}
	\Prob(d M) = \frac{1}{Z_n} \exp(- \tr Q(M)) dM
\end{equation}
on the space of $n \times n$ real symmetric matrices, where $dM$ refers to the Lebesgue measure on the $n(n+1)/2$ algebraically independent entries of $M$, $Q$ is a real-valued function, and $Z_n$ is the normalization constant
$$ Z_n := \int e^{-\tr Q(M)} dM. $$
Here $Q(M)$ is defined by the functional calculus for $M$.
\end{definition}

The prototypical example of an orthogonally invariant ensemble is the \emph{Gaussian orthogonal ensemble} (GOE).  The GOE is defined by the probability distribution in \eqref{eq:oe} when $Q(x) = \frac{1}{4} x^2$.  In this case, a matrix drawn from the GOE is actually a Wigner random matrix.  Indeed, if $M = (m_{ij})_{i,j=1}^n$ is drawn from the GOE, the elements $\{ m_{ij} : 1 \leq i \leq j \leq n \}$ are independent Gaussian random variables with mean zero and variance $1+\delta_{ij}$, where 
$$ \delta_{ij} := \left\{
     \begin{array}{lr}
       1, & \text{if } i = j, \\
       0, & \text{if } i \neq j.
     \end{array} 
     \right. $$
We refer the reader to \cite{DG} for further examples and properties of orthogonally invariant ensembles.

\section{Main results}\label{sec:mainresults}
In this section, we present the main results of this work. To improve readability, we present the results first, and we postpone the technical details and proofs until the subsequent sections. Our main results below address Question \ref{quest:main} for the random matrix ensembles introduced above.  In particular, we will consider two distinct cases: when $b$ is deterministic and when $b$ is a random vector.

\subsection{Results for Wigner matrices and deterministic vectors} \label{sec:wigdet}

We first consider the controllability of $(W,e_i)$ when $W$ is a Wigner matrix.  

\begin{theorem}[Minimal controllability of Wigner matrices] \label{thm:wigner}
Let $\xi$ be a sub-gaussian random variable with mean zero and unit variance; let $\zeta$ be a (possibly degenerate) sub-gaussian random variable with mean zero.  Assume $W$ is an $n \times n$ Wigner random matrix with atom variables $\xi$ and $\zeta$.  Let $F$ be an $n \times n$ deterministic real symmetric matrix such that $\|F\| \leq n^{\gamma}$ for some fixed parameter $\gamma > 0$.  Then, for any $\alpha > 0$, there exists $C > 0$ (depending only on $\alpha$, $\gamma$, $\xi$, and $\zeta$) such that $(W+F, e_i)$ is controllable for every $1 \leq i \leq n$ with probability at least $1 - C n^{-\alpha}$.
\end{theorem}

\begin{remark}
The matrix $M := W + F$ can be viewed as a Wigner matrix whose entries are not required to have mean zero.  Alternatively, it is often useful to view $M$ as a random perturbation of the matrix $F$.  
\end{remark}

As a consequence of Theorem \ref{thm:wigner}, we obtain the following corollary for the adjacency matrix of a random graph.

\begin{corollary}[Minimal controllability of random graphs] \label{cor:adj}
Fix $0 < p < 1$ and $\alpha > 0$.  Let $A$ be the adjacency matrix of $G(n,p)$.  Then there exists $C > 0$ (depending only on $p$ and $\alpha$) such that $(A, e_i)$ is controllable for each $1 \leq i \leq n$ with probability at least $1 - C n^{-\alpha}$.
\end{corollary}
As a result of this corollary, Conjecture~\ref{conj:conjecture1} follows immediately. This result also suggests that for large systems, the difficult problem of minimal controllability \cite{alex} is \textit{most probably} a trivial problem. 

\subsection{Results for Wigner matrices and random vectors} \label{sec:wigrand}
We now consider the case when $b$ is a random vector.  In particular, we show below (see Corollary \ref{cor:adjrand}) that $(A,b)$ is controllable with high probability when $A$ is the adjacency matrix of a random graph $G(n,p)$ and $b$ is a random binary vector. 

We begin, as before, with a result for a very general class of Wigner matrices.  Indeed, Theorem \ref{thm:wignerrand} below shows that, with high probability, $(W,b)$ is controllable when $W$ is a Wigner matrix and $b$ is a random vector, independent of $W$.  More generally, we also consider deterministic perturbations of $W$ and $b$.  That is, we show $(W+F, b + \mu)$ is controllable with high probability, where $F$ is a deterministic matrix and $\mu$ is a deterministic vector.  
\begin{theorem} \label{thm:wignerrand}
Let $\xi$ be a sub-gaussian random variable with mean zero and unit variance; let $\zeta$ be a (possibly degenerate) sub-gaussian random variable with mean zero.  Assume $W$ is an $n \times n$ Wigner random matrix with atom variables $\xi$ and $\zeta$.  Let $F$ be an $n \times n$ deterministic real symmetric matrix such that $\|F\| \leq n^{\gamma}$ for some $\gamma > 0$.  Suppose $b$ is an independent random vector whose coordinates are iid copies of $\xi$; let $\mu$ be a deterministic vector.  Then, for any $\alpha > 0$, there exists $C > 0$ (depending only on $\alpha$, $\gamma$, $\xi$, and $\zeta$) such that $(W+F, b+\mu)$ is controllable with probability at least $1 - C n^{-\alpha}$.
\end{theorem}

Another interesting problem is the controllability of the pair $(W+F,u)$ when $u$ is uniformly distributed over the unit sphere, independent of $W$.  
\begin{theorem} \label{thm:wignerunif}
Let $\xi$ be a sub-gaussian random variable with mean zero and unit variance; let $\zeta$ be a (possibly degenerate) sub-gaussian random variable with mean zero.  Assume $W$ is an $n \times n$ Wigner random matrix with atom variables $\xi$ and $\zeta$.  Let $F$ be an $n \times n$ deterministic real symmetric matrix such that $\|F\| \leq n^{\gamma}$ for some $\gamma > 0$.  Suppose $u$ is an independent random vector uniformly distributed on the unit sphere in $\mathbb{R}^n$.  Then, for any $\alpha > 0$, there exists $C > 0$ (depending only on $\alpha$, $\gamma$, $\xi$, and $\zeta$) such that $(W+F, u)$ is controllable with probability at least $1 - C n^{-\alpha}$.
\end{theorem}

As a consequence of the previous two theorems, we have the following interesting result on the controllability of random graphs. 
\begin{corollary} \label{cor:adjrand}
Fix $0 < p < 1$ and $\alpha > 0$.  Let $A$ be the adjacency matrix of $G(n,p)$.  Let $b$ be an independent random vector whose entries are iid zero-one Bernoulli random variables with parameter $p$, and let $u$ be an independent random vector uniformly distributed on the unit sphere in $\mathbb{R}^n$.  Then there exists $C > 0$ (depending only on $p$ and $\alpha$) such that $(A, b)$ and $(A,u)$ are controllable with probability at least $1 - C n^{-\alpha}$.
\end{corollary}

\subsection{Results for orthogonally invariant ensembles} \label{sec:oe}
Due to the continious nature of the orthogonally invariant ensembles and the singular nature of controllability, we prove stronger universality results for random matrices drawn from this ensemble.  
\begin{theorem}[Controllability of any orthogonally invariant ensemble] \label{thm:oe}
Let $M$ be an $n \times n$ random real symmetric matrix drawn from an orthogonally invariant ensemble.  Let $b$ be any nonzero vector in $\mathbb{R}^n$.  Then $(M,b)$ is controllable with probability one.
\end{theorem}

As a consequence of the above result, we immediately obtain the following corollary.  
\begin{corollary} \label{cor:oe}
Let $M$ be an $n \times n$ random real symmetric matrix drawn from an orthogonally invariant ensemble.  Let $b$ be an independent random vector in $\mathbb{R}^n$ (with any distribution).  Then $(M,b)$ is controllable with probability $1 - \Prob( b = 0)$.
\end{corollary}

\section{Tools} \label{sec:tools}
We now present the tools we will need to prove our main results.
\subsection{Tools from linear algebra}
Recall that the spectrum of an $n \times n$ Hermitian matrix $A$ is said to be \emph{simple} if all of its $n$ eigenvalues are distinct; that is,
$$ \lambda_1(A) < \cdots < \lambda_n(A). $$

In the sequal, we make use of the following criteria for controllability, which is based on the Popov--Belevitch--Hautus Controllability Test (see, for instance, \cite[Theorem 12.3]{hespanha}).  

\begin{lemma}[Controllability criteria] \label{lemma:simplenz}
Let $A$ be an $n \times n$ Hermitian matrix, and let $b$ be any nonzero vector.  If $A$ has simple spectrum and $v_j(A)^\ast b \neq 0$ for all $1 \leq j \leq n$, then $(A,b)$ is controllable.
\end{lemma}
\begin{remark}
In general, the eigenvectors $v_1(A), \ldots, v_n(A)$ of $A$ are not uniquely determined.  However, if the spectrum of $A$ is simple, then the eigenvectors of $A$ are uniquely determined up to phase.  In particular, the condition $v_j(A)^\ast b \neq 0$ is not affected by this choice of phase.  
\end{remark}
\begin{proof}[Proof of Lemma \ref{lemma:simplenz}]
In order to reach a contradiction, assume the matrix
$$ B := \begin{pmatrix} b & Ab &\cdots& A^{n-1}b \end{pmatrix}$$
does not have full rank.  Then there exists a nonzero vector $a = (a_j)_{j=0}^{n-1}$ such that $B a = 0$.

For notational convenience, let $\lambda_1 < \cdots < \lambda_n$ be the eigenvalues of $A$ with corresponding orthonormal eigenvectors $v_1, \ldots, v_n$.  By the spectral theorem, for any $k \in \mathbb{N}$,
$$ A^k b = \sum_{j=1}^n (\lambda_j^k v_j^\ast b) v_j. $$
Thus, using the notation that $A^0$ is the identity matrix, we have
\begin{align*}
	0 = B a &= \sum_{k=0}^{n-1} a_k A^k b = \sum_{k=1}^{n-1} a_k \sum_{j=1}^n \lambda_j^k v_j^\ast b v_j = \sum_{j=1}^n v_j \left( \sum_{k=0}^{n-1} v_j^\ast b \lambda_j^k a_k \right) = \sum_{j=1}^n \beta_j v_j,
\end{align*}
where
$$ \beta_j := \sum_{k=0}^{n-1} v_j^\ast b \lambda_j^k a_k. $$

Since $v_1, \ldots, v_n$ are linearly independent, it must be the case that $\beta_j = 0$ for all $1 \leq j \leq n$.  We observe that
\begin{align*}
	\beta_j &= v_j^\ast b \sum_{k=0}^{n-1} \lambda_j^k a_k =  v_j^\ast b \begin{pmatrix} \lambda_j^{0} \\ \lambda_j \\ \vdots \\ \lambda_j^{n-1} \end{pmatrix}^{\mathrm{T}} a.
\end{align*}
Since $v_j^\ast b \neq 0$ for all $1 \leq j \leq n$ by assumption, it must be the case that
$$ \begin{pmatrix} \lambda_j^{0} \\ \lambda_j \\ \vdots \\ \lambda_j^{n-1} \end{pmatrix}^{\mathrm{T}} a = 0 $$
for all $1 \leq j \leq n$.  However, this implies that the matrix
$$ \Lambda := \begin{pmatrix} \lambda_1^{0} & \lambda_2^{0} & \cdots & \lambda_n^{0} \\ \lambda_1 & \lambda_2 & \cdots & \lambda_n \\ \vdots & \vdots & \ddots & \vdots \\ \lambda_1^{n-1} & \lambda_2^{n-1} & \cdots & \lambda_n^{n-1} \end{pmatrix}^{\mathrm{T}} $$
is singular.  As $\Lambda$ is a Vandermonde matrix, we find that
$$ \det \Lambda = \prod_{i < j} (\lambda_j - \lambda_i) = 0. $$
This implies that the spectrum of $A$ is not simple, a contradiction.
\end{proof}

Cauchy's interlacing theorem implies that the eigenvalues of the $i$th minor of a Hermitian matrix interlace with the eigenvalues of the original matrix; see \cite[Corollary III.1.5]{B} for details. One of our main observations is that, if this interlacing is strict, then the pair $(A,e_i)$ is controllable. For this, we will need the following result from \cite{ESY}; see also \cite[Lemma 41]{TVuniv}.

\begin{lemma}[\cite{ESY}; Lemma 41 from \cite{TVuniv}] \label{lemma:coordinate}
Let
$$ A_n = \begin{pmatrix} A_{n-1} & X \\ X^\ast & d \end{pmatrix} $$
be an $n \times n$ Hermitian matrix, where $A_{n-1}$ is the upper-left $(n-1) \times (n-1)$ minor of $A_n$, $X \in \mathbb{C}^{n-1}$, and $d \in \mathbb{R}$.  Let $\begin{pmatrix} v \\ x \end{pmatrix}$ be a unit eigenvector of $\lambda_i(A_n)$, where $v \in \mathbb{C}^{n-1}$ and $x \in \mathbb{C}$.  Suppose that none of the eigenvalues of $A_{n-1}$ are equal to $\lambda_i(A_n)$.  Then
$$ |x|^2 = \frac{1}{1 + \sum_{j=1}^{n-1} (\lambda_j(A_{n-1}) - \lambda_i(A_n))^{-2} |v_j(A_{n-1})^\ast X|^2}, $$
where $v_j(A_{n-1})$ is a unit eigenvector corresponding to the eigenvalue $\lambda_j(A_{n-1})$. In particular, under the assumptions above, $x \neq 0$.
\end{lemma}

The above lemma implies that if $\lambda_j(A_{n-1}) - \lambda_i(A_n)\not=0$ for all $i,j$, then $e_n$ is not orthogonal to any eigenvector of $A_n$.  The choice of the $n$th minor (and the $n$th entry of an eigenvector) is arbitrary; indeed, the result can be applied to any minor by conjugating the matrix $A_n$ by a permutation matrix.  

We will use the following lemma in order to show that $\lambda_j(A_{n-1}) - \lambda_i(A_n)\not=0$ for all $i,j$.  
\begin{lemma} \label{lemma:nonzero}
Let
$$ A_n = \begin{pmatrix} A_{n-1} & X \\ X^\ast & d \end{pmatrix} $$
be an $n \times n$ Hermitian matrix, where $A_{n-1}$ is the upper-left $(n-1) \times (n-1)$ minor of $A_{n}$, $X \in \mathbb{C}^{n-1}$, and $d \in \mathbb{R}$.  If an eigenvalue of $A_n$ corresponds with an eigenvalue of $A_{n-1}$, then there exists a unit eigenvector $w$ of $A_{n-1}$ such that $X^\ast w = 0$.
\end{lemma}
\begin{proof}
Assume $\lambda$ is an eigenvalue of $A_{n}$ with unit eigenvector $\begin{pmatrix} v \\ q \end{pmatrix}$, where $v \in \mathbb{C}^{n-1}$ and $q \in \mathbb{C}$.  In addition, assume $\lambda$ is an eigenvalue of $A_{n-1}$ with unit eigenvector $u \in \mathbb{C}^{n-1}$.  From the eigenvalue equation
$$ A_n \begin{pmatrix} v \\ q \end{pmatrix} = \lambda \begin{pmatrix} v \\ q \end{pmatrix}, $$
we obtain
\begin{align}
	A_{n-1} v + q X &= \lambda v, \label{eq:eig1} \\
	X^\ast v + qd &= \lambda q. \label{eq:eig2}
\end{align}
Since $u^\ast A_{n-1} = \lambda u^\ast$, we multiply \eqref{eq:eig1} on the left by $u^\ast$ to obtain $q u^\ast X = 0$.  In other words, either $u^\ast X = 0$ or $q = 0$.

If $u^\ast X = 0$, then the proof is complete (since $u$ is a unit eigenvector of $A_{n-1}$).  Assume $q = 0$.  Then $v$ is a unit vector, and equations \eqref{eq:eig1} and \eqref{eq:eig2} imply that
\begin{align*}
	A_{n-1} v &= \lambda v, \\
	X^\ast v &= 0.
\end{align*}
In other words, $v$ is a unit eigenvector of $A_{n-1}$ and $X^\ast v = 0$.
\end{proof}

\subsection{Tools from probability theory}

We present some probabilistic results we will need to prove our main results.  The first lemma below is a standard result; we include a proof for completeness.

\begin{lemma} \label{lemma:uniform}
Let $v \in \mathbb{R}^n$ be a unit vector, and assume $u$ is a random vector uniformly distributed on the unit sphere in $\mathbb{R}^n$.  Then
$$ \Prob \left( v^\mathrm{T} u = 0 \right) = 0. $$
\end{lemma}
\begin{proof}
Observe that for any $n \times n$ orthogonal matrix $U$, $v^\mathrm{T} u = (Uv)^{\mathrm{T}} (Uu)$ and $Uu$ has the same distribution as $u$.  Thus, it suffices to show that
$$ \Prob \left( e_1^\mathrm{T} u = 0\right) = 0. $$
We note that $u$ has the same distribution as
$$ \frac{1}{\sqrt{ \sum_{k=1}^n \xi_k^2}} \left( \xi_1, \ldots, \xi_n \right), $$
where $\xi_1, \ldots, \xi_n$ are iid standard normal random variables.  Indeed, this follows since the normal vector $(\xi_1, \ldots, \xi_n)$ is rotationally invariant.  As $\sum_{k=1}^n \xi_k^2 > 0$ with probability one, we find that
$$ \Prob \left( e_1^\mathrm{T} u = 0 \right) = \Prob \left( \xi_1 = 0 \right) = 0, $$
and the proof is complete.
\end{proof}

We will also need the following bound for the spectral norm of a Wigner matrix.

\begin{lemma}[Spectral norm of a Wigner matrix] \label{lemma:norm}
Let $\xi$ be a sub-gaussian random variable with mean zero and unit variance; let $\zeta$ be a (possibly degenerate) sub-gaussian random variable with mean zero.  Assume $W$ is an $n \times n$ Wigner random matrix with atom variables $\xi$ and $\zeta$.  Then there exists constants $C_0, c_0 > 0$ (depending only on the sub-gaussian moments of $\xi$ and $\zeta$) such that $\|W\| \leq C_0 \sqrt{n}$ with probability at least $1 - C_0 \exp(-c_0 n)$.
\end{lemma}
\begin{proof}
It follows from \cite[Proposition 5.10]{V} that, for any unit vectors $u,v \in \mathbb{R}^n$ and any $t > 0$, 
$$ \Prob \left( |u^\mathrm{T} W v | \geq t \right) \leq C \exp(-c t^2), $$
where $C,c > 0$ depend only on the sub-gaussian moments of $\xi$ and $\zeta$.  Thus, by a standard net argument (see \cite[Fact 11]{OVW} for details), we obtain
$$ \Prob \left( \|W \| \geq C' \sqrt{n} \right) \leq C' \exp(-c ' n) $$
for some constants $C', c' > 0$ depending only on the sub-gaussian moments of $\xi$ and $\zeta$.  
\end{proof}

\section{Proofs of the Main Results}\label{sec:proof}
\subsection{Proof of results in Section \ref{sec:wigdet}} \label{sec:wigpf}

This section is devoted to the proof of Theorem \ref{thm:wigner} and Corollary \ref{cor:adj}.  In view of Lemma \ref{lemma:simplenz}, it suffices to show that, with high probability, the spectrum of $M:=W + F$ is simple and $v_j(M)^\mathrm{T} e_i \neq 0$ for all $1 \leq j \leq n$.  It was recently shown in \cite{NTV} that $M$ has simple spectrum with high probability.

\begin{theorem}[Theorem 2.6 from \cite{NTV}] \label{thm:simple}
Let $\xi$ be a sub-gaussian random variable with mean zero and unit variance; let $\zeta$ be a (possibly degenerate) sub-gaussian random variable with mean zero.  Assume $W$ is an $n \times n$ Wigner random matrix with atom variables $\xi$ and $\zeta$.  Let $F$ be an $n \times n$ deterministic real symmetric matrix such that $\|F\| \leq n^{\gamma}$ for some $\gamma > 0$.  Then, for any $\alpha > 0$, there exists $C > 0$ (depending only on $\alpha$, $\gamma$, $\xi$, and $\zeta$) such that $W + F$ has simple spectrum with probability at least $1 - Cn^{-\alpha}$.
\end{theorem}
\begin{remark}
\cite[Theorem 2.6]{NTV} is actually more general than what is stated above.  Indeed, the results in \cite{NTV} also control the gaps between the eigenvalues $\lambda_{i+1}(W+F) - \lambda_{i}(W+F)$ for all $1 \leq i \leq n-1$.
\end{remark}

We now verify the following result for the eigenvectors of $M:=W+F$.

\begin{lemma} \label{lemma:eigenvectors}
Let $\xi$ be a sub-gaussian random variable with mean zero and unit variance; let $\zeta$ be a (possibly degenerate) sub-gaussian random variable with mean zero.  Assume $W$ is an $n \times n$ Wigner random matrix with atom variables $\xi$ and $\zeta$.  Let $F$ be an $n \times n$ deterministic real symmetric matrix such that $\|F\| \leq n^{\gamma}$ for some $\gamma > 0$.  Then, for any $\alpha > 0$, there exists $C > 0$ (depending only on $\alpha$, $\gamma$, $\xi$, and $\zeta$) such that,
$$ \Prob \left( \exists i,j \in \{1, \ldots, n\} \text{ such that } v_j(M)^\mathrm{T} e_i = 0 \right) \leq C n^{-\alpha}, $$
where $M:=W+F$.
\end{lemma}

In order to prove Lemma \ref{lemma:eigenvectors}, we will need the following notation and results.  Let $x = (x_k)_{k=1}^n \in \mathbb{R}^n$ and let $\xi$ be a real-valued random variable.  Consider the \emph{small ball probability}
\begin{equation} \label{eq:def:smallball}
	\rho_{\xi,\delta}(x) := \sup_{a \in \mathbb{R}} \Prob \left( |\xi_1 x_1 + \cdots \xi_n x_n - a | \leq \delta \right),
\end{equation}
where $\xi_1, \ldots, \xi_n$ are iid copies of $\xi$.  In particular, even if the vector $x$ is random, we assume $\xi_1, \ldots, \xi_n$ are independent of $x$ and the probability above is taken with respect to the random variables $\xi_1, \ldots, \xi_n$.  In other words, we condition on (or freeze) the coordinates of $x$ when we compute the small ball probability in \eqref{eq:def:smallball}.

We will need the following result concerning the small ball probability when $x$ happens to be an eigenvector of a Wigner matrix.

\begin{theorem}[Theorem 4.3 from \cite{NTV}] \label{thm:smallball}
Let $\xi$ be a sub-gaussian random variable with mean zero and unit variance; let $\zeta$ be a (possibly degenerate) sub-gaussian random variable with mean zero.  Assume $W$ is an $n \times n$ Wigner random matrix with atom variables $\xi$ and $\zeta$.  Let $F$ be an $n \times n$ deterministic real symmetric matrix such that $\|F\| \leq n^{\gamma}$ for some $\gamma > 0$.  Then, for any $\alpha > 0$, there exists $\beta, \beta_0, C, C_0 > 0$ (depending only on $\alpha$, $\gamma$, $\xi$, and $\zeta$) such that the following holds with probability at least $1 - C_0 \exp(-\beta_0 n)$: every unit eigenvector $v$ of $W + F$ obeys the anti-concentration estimate
$$ \rho_{\xi, n^{-\beta}}(v) \leq C n^{-\alpha}. $$
\end{theorem}

The proof of Theorem \ref{thm:smallball} presented in \cite{NTV} is based on the \emph{inverse Littlewood--Offord theory} introduced in \cite{TVlo}; see \cite{NV} for a survey.   In broad terms, Nguyen, Tao, and Vu establish Theorem \ref{thm:smallball} using the following procedure.  It is first shown that if $\rho_{\xi, \delta}(x) > C n^{-\alpha}$, then the vector $x$ must have a rich additive structure.  Intuitively, this follows since the sum
$$ \sum_{k=1}^n \xi_k x_k $$
can only be concentrated on a given value when most of the coefficients $x_k$ are arithmetically well comparable.  On the other hand, if we take $x$ to be an eigenvector of a Wigner random matrix, then one expects the vector $x$ to look random and not have any rigid structure.  This explains the intuition behind Theorem \ref{thm:smallball}.

We also note that Theorem \ref{thm:smallball} is a slight extension of \cite[Theorem 4.3]{NTV}.  We have not included the assumption that the eigenvalues of $W+F$ are no larger than some polynomial in $n$ since this follows immediately from the bound $\|W + F\| \leq \|W\| + n^{\gamma}$ and Lemma \ref{lemma:norm}.

We now prove Lemma \ref{lemma:eigenvectors}.

\begin{proof}[Proof of Lemma \ref{lemma:eigenvectors}]
Set $M := W + F$, and let $\alpha > 0$.  Fix $1 \leq i, j \leq n$.  We will show that
$$ \Prob \left( v_j(M)^\mathrm{T} e_i = 0 \right) \leq C n^{-\alpha - 2}. $$
The claim will then follow from the union bound over $1 \leq i,j \leq n$.

Let $M'$, $W'$, $F'$ denote the $(n-1) \times (n-1)$ sub-matrices of $M$, $W$, $F$ formed by removing the $i$th row and $i$th column.  Let $X, Y, Z$ denote the $i$th columns of $M, W, F$ with the $i$th entries removed.  In particular, $M' = W' + F'$ and $X = Y + Z$.  In addition, the entries of $Y$ are iid copies of $\xi$.

Suppose $v_j(M)^\mathrm{T} e_i = 0$.  Then, by Lemma \ref{lemma:coordinate}, it must be the case that one of the eigenvalues of $M'$ corresponds to $\lambda_j(M)$.  Thus, by Lemma \ref{lemma:nonzero}, there exists a unit eigenvector $v$ of $M'$ such that $X^\mathrm{T}v = 0$.  Therefore, it suffices to show that
$$ \Prob \left( \text{there exists a unit eigenvector } v \text{ of } M' \text{ such that } X^\mathrm{T}v = 0 \right) \leq C n^{-\alpha - 2}. $$

Observe that $M'$ and $X$ are independent.  In other words, conditioning on $M'$ does not affect $X$.  Also, recall the definition of the small ball probability given in \eqref{eq:def:smallball}.  Let $\beta, C$ be positive constants to be chosen later.  For a fixed realization of $M'$, we say that $M'$ is in a $(\beta,C)$-good configuration if
$$ \rho_{\xi, n^{-\beta}}(v) \leq C n^{-\alpha - 2} $$
for all unit eigenvectors $v$ of $M'$.  Let $\Omega_{\beta,C}$ be the event where the spectrum of $M'$ is simple and $M'$ is in a $(\beta,C)$-good configuration. Theorems \ref{thm:simple} and \ref{thm:smallball} imply that there exists $\beta, C > 0$ such that $\Prob \left( \Omega_{\beta,C}^c \right) \leq C n^{-\alpha-3}$.

For notational convenience, let $S$ denote the set of unit eigenvectors $M'$.  When the eigenvalues of $M'$ are simple, the (unit) eigenvectors of $M'$ are uniquely determined up to sign.  Thus, on the event $\Omega_{\beta,C}$, $S$ is a set consisting of $n-1$ orthonormal vectors (where the choice of sign for each vector is arbitrary).  Moreover, for a fixed realization of $M'$ and for a fixed vector $v \in S$, $X^\mathrm{T} v = 0$ implies that $|Y^\ast v - a| \leq n^{-\beta}$ for some $a \in \mathbb{R}$.  In other words,
$$ \Prob \left( X^\mathrm{T} v = 0 \right) \leq \rho_{\xi, n^{-\beta}}(v). $$
Thus, by conditioning on $\Omega_{\beta,C}$, we apply the union bound over all vectors in $S$ to obtain
\begin{align*}
	\Prob \left( \exists v \in S \text{ such that } X^\mathrm{T}v = 0 \right) &\leq \Prob \left( \exists v \in S \text{ such that } X^\mathrm{T}v = 0 | \Omega_{\beta,C} \right) + \Prob(\Omega_{\beta,C}^c) \\
		&\leq C n n^{-\alpha - 3} + C n^{-\alpha - 3}
\end{align*}
and the proof of the lemma is complete.
\end{proof}

We are now ready to prove Theorem \ref{thm:wigner}

\begin{proof}[Proof of Theorem \ref{thm:wigner}]
Set $M := W + F$, and let $\alpha > 0$.  Define the event
$$ \Omega := \{ \text{the spectrum of } M \text{ is simple} \} \bigcap \{ v_j(M)^\mathrm{T} e_i \neq 0 \text{ for all } 1 \leq i, j \leq n \}. $$
By Theorem \ref{thm:simple} and Lemma \ref{lemma:eigenvectors}, there exists $C > 0$ such that $\Prob (\Omega^c) \leq C n^{-\alpha}$.  On the event $\Omega$, it follows from Lemma \ref{lemma:simplenz} that $(M,e_i)$ is controllable for each $1 \leq i \leq n$.
\end{proof}

Corollary \ref{cor:adj} now follows from Theorem~\ref{thm:wigner}.

\begin{proof}[Proof of Corollary \ref{cor:adj}]
Let $0 < p < 1$, and let $A$ be the adjacency matrix of $G(n,p)$.  Set $\sigma^2 := p(1-p)$.  Note that $(A,e_i)$ is controllable if and only if $(\sigma^{-1} A, e_i)$  is controllabe.

Define the random variables $\zeta := 0$ and
$$ \xi := \left\{
     \begin{array}{rl}
       \frac{1-p}{\sigma}, & \text{with probability } p, \\
       -\frac{p}{\sigma}, & \text{with probability } 1-p.
     \end{array}
   \right. $$
In particular, $\xi$ and $\zeta$ both have mean zero and $\xi$ has unit variance.  In addition, both are sub-gaussian random variables (with sub-gaussian moments depending only on $p$).  Let $W$ be an $n \times n$ Wigner random matrix with atom variables $\xi$ and $\zeta$.  Define the deterministic matrix $F = (f_{ij})_{i,j=1}^n$ by
$$ f_{ij} := \left\{
     \begin{array}{rl}
       \frac{p}{\sigma}, & \text{if } i \neq j, \\
       0, & \text{if } i = j.
     \end{array}
   \right. $$
Then $\sigma^{-1} A$ has the same distribution as $W+F$.  Thus, it suffices to estimate the probability that $(W+F,e_i)$ is controllable for all $1 \leq i \leq n$.  Hence, the proof is now complete by applying Theorem \ref{thm:wigner}.
\end{proof}

\subsection{Proof of results in Section \ref{sec:wigrand}} \label{sec:wigrandpf}

This section is devoted to the proof of Theorem \ref{thm:wignerrand}, Theorem \ref{thm:wignerunif}, and Corollary \ref{cor:adjrand}.  Unsurprisingly, the proofs presented here are very similar to the proofs of the previous section.

\begin{proof}[Proof of Theorem \ref{thm:wignerrand}]
Set $M := W + F$, and let $\alpha > 0$.  In view of Lemma \ref{lemma:simplenz} and Theorem \ref{thm:simple}, it suffices to show that
$$ \Prob \left( \text{there exists a unit eigenvector } v \text{ of } M \text{ such that } v^\mathrm{T} (b+\mu) = 0 \right) \leq C n^{-\alpha}. $$

We proceed as in the proof of Lemma \ref{lemma:eigenvectors}.  For a fixed realization of $M$, we say that $M$ is in a $(\beta,C)$-good configuration if
$$ \rho_{\xi, n^{-\beta}}(v) \leq C n^{-\alpha - 1} $$
for all unit eigenvectors $v$ of $M$; here $\rho_{\xi, n^{-\beta}}(v)$ denotes the small ball probability defined in \eqref{eq:def:smallball}.  Let $\Omega_{\beta,C}$ be the event where the spectrum of $M$ is simple and $M$ is in a $(\beta,C)$-good configuration. Theorems \ref{thm:simple} and \ref{thm:smallball} imply that there exists $\beta, C > 0$ such that $\Prob \left( \Omega_{\beta,C}^c \right) \leq C n^{-\alpha-1}$.

For notational convenience, let $S$ denote the set of unit eigenvectors $M$.  When the eigenvalues of $M$ are simple, the eigenvectors of $M$ are uniquely determined up to sign.  Thus, on the event $\Omega_{\beta,C}$, $S$ is a set consisting of $n-1$ orthonormal vectors (where the choice of sign for each vector is arbitrary).  Moreover, for a fixed realization of $M$ and for a fixed vector $v \in S$, $v^\mathrm{T} (b+\mu)= 0$ implies that $|v^\mathrm{T} b - a| \leq n^{-\beta}$ for some $a \in \mathbb{R}$.  In other words,
$$ \Prob \left( v^\mathrm{T} (b+\mu) = 0 \right) \leq \rho_{\xi, n^{-\beta}}(v). $$
Here we have used the assumption that $b$ is independent of $M$.  Thus, by conditioning on $\Omega_{\beta,C}$, we apply the union bound over all vectors in $S$ to obtain
\begin{align*}
	\Prob \left( \exists v \in S \text{ with } v^\mathrm{T} (b + \mu) = 0 \right) &\leq \Prob \left( \exists v \in S \text{ with } v^\mathrm{T} (b + \mu) = 0 | \Omega_{\beta,C} \right) + \Prob(\Omega_{\beta,C}^c) \\
		&\leq C n n^{-\alpha - 1} + C n^{-\alpha - 1}
\end{align*}
and the proof is complete.
\end{proof}

\begin{proof}[Proof of Theorem \ref{thm:wignerunif}]
Set $M := W+F$, and let $\alpha > 0$.  By Lemma \ref{lemma:simplenz}, it suffices to show that, with probability at least $1 - Cn^{-\alpha}$, the spectrum of $M$ is simple and that $v_j(M)^\mathrm{T} u \neq 0$ for all $1 \leq j \leq n$.  The claims follow from Theorem \ref{thm:simple} and Lemma \ref{lemma:uniform}.
\end{proof}

Using Theorems \ref{thm:wignerrand} and \ref{thm:wignerunif}, we now prove Corollary~\ref{cor:adjrand}.  
\begin{proof}[Proof of Corollary \ref{cor:adjrand}]
We proceed as in the proof of Corollary \ref{cor:adj}.  Let $\alpha > 0$.  Let $0 < p < 1$, and let $A$ be the adjacency matrix of $G(n,p)$.  Assume $b$ is a random vector whose entries are iid zero-one Bernoulli random variables with parameter $p$.  Set $\sigma^2 := p(1-p)$.  Observe that $(A,b)$ is controllable if and only if $(\sigma^{-1} A, b)$ is controllable.  Thus, it suffices to compute the probability that $(\sigma^{-1} A, b)$ is controllable.

Define the random variables $\zeta := 0$ and
$$ \xi := \left\{
     \begin{array}{rl}
       \frac{1-p}{\sigma}, & \text{with probability } p, \\
       -\frac{p}{\sigma}, & \text{with probability } 1-p.
     \end{array}
   \right. $$
In particular, $\xi$ and $\zeta$ both have mean zero and $\xi$ has unit variance.  In addition, both are sub-gaussian random variables (with sub-gaussian moments depending only on $p$).  Let $W$ be an $n \times n$ Wigner random matrix, independent of $b$, with atom variables $\xi$ and $\zeta$.  Define the deterministic matrix $F = (f_{ij})_{i,j=1}^n$ by
$$ f_{ij} := \left\{
     \begin{array}{rl}
       \frac{p}{\sigma}, & \text{if } i \neq j, \\
       0, & \text{if } i = j.
     \end{array}
   \right. $$
Then $\sigma^{-1} A$ has the same distribution as $W+F$.  Thus, it suffices to estimate the probability that $(W+F,b)$ is controllable.  Since rank is invariant under scaling, it in fact suffices to compute the probability that $(W+F, \sigma^{-1} b)$ is controllable.

Let $b'$ be a random vector, independent of $W$, whose entries are iid copies of $\xi$; let $\mu$ be a vector in $\mathbb{R}^n$ whose entries are all taken to be $p/\sigma$.  Then, $\sigma^{-1}b$ has the same distribution as $b' + \mu$.  Therefore it suffices to estimate the probability that $(W+F, b' + \mu)$ is controllable.  By Theorem \ref{thm:wignerrand}, we conclude that $(W+F,b'+\mu)$ is controllable with probability at least $1-Cn^{-\alpha}$.

We now show that $(A,u)$ is controllable, where $u$ is a random vector uniformly distributed on the unit sphere, independent of $A$.  By the arguments above, it suffices to estimate the probability that $(W+F, u)$ is controllable, where $W$ and $F$ are defined above and $W, u$ are taken to be independent.  The claim now follows from Theorem \ref{thm:wignerunif}.
\end{proof}

\subsection{Proof of results in Section \ref{sec:oe}} \label{sec:oepf}

This section is devoted to the proof of Theorem \ref{thm:oe} and Corollary \ref{cor:oe}.  
\begin{proof}[Proof of Theorem \ref{thm:oe}]
Let $M$ be an $n \times n$ real symmetric matrix drawn from an orthogonally invariant ensemble, and let $b$ be a nonzero vector in $\mathbb{R}^n$.  Since the rank of a matrix is invariant under scaling, it suffices to assume that $b$ is a unit vector.

Let $U$ be an $n \times n$ random orthogonal matrix distributed according to Haar measure, independent of $M$.  Since $M$ is orthogonally invariant (see, for example, \cite[Section 2.2]{DG}), it follows that $UMU^\ast$ has the same distribution as $M$.  Thus,
\begin{align*}
	\Prob &\left( \rank \begin{pmatrix} b & Mb &\cdots & M^{n-1} b \end{pmatrix} = n \right) \\
	&\qquad\qquad\qquad = \Prob \left( \rank \begin{pmatrix} U U^\ast b & U M U^\ast b & \cdots & U M^{n-1} U^\ast b \end{pmatrix} = n \right).
\end{align*}
As $U$ is an invertible matrix, it suffices to show that $(M, U^\ast b)$ is controllable with probability one.

Since $U^\ast b$ is a random vector uniformly distributed on the unit sphere in $\mathbb{R}^n$, it in fact suffices to show that $(M,u)$ is controllable with probability one, where $u$ is uniformly distributed on the unit sphere, independent of $M$.

It is well-known (see, for instance, \cite[Section 2.3]{DG}), that the spectrum of $M$ is simple with probability one.  By Lemma \ref{lemma:simplenz}, it thus suffices to show that, with probability one, $v_j(M)^\mathrm{T} u \neq 0$ for all $1 \leq j \leq n$.  The claim now follows from Lemma \ref{lemma:uniform}.
\end{proof}

\begin{proof}[Proof of Corollary \ref{cor:oe}]
Let $\Omega$ be the event that $(M,b)$ is controllable.  We condition on $b$ (since $M$ and $b$ are independent) to obtain
\begin{align*}
	\Prob \left( \Omega \right) &= \Prob \left( \Omega | b \neq 0 \right) \Prob (b \neq 0) + \Prob \left( \Omega | b = 0 \right) \Prob( b = 0 ) \\
	& = \Prob( b \neq 0 ) \\
	&= 1 - \Prob(b = 0)
\end{align*}
by Theorem \ref{thm:oe}.  Here we used the fact that $(A,0)$ is not controllable for any matrix $A$.
\end{proof}

\section*{Acknowledgments} 
The authors are thankful Prof.\ Soheil Mohajer for his constructive discussions as well as suggestions to improve this work. We are thankful to Prof.\ Alex Olshevsky for communicating Conjecture~\ref{conj:conjecture1} to us.  We also thank Prof.\ Van Vu for useful discussions.

\section{Closing remarks} \label{sec:closing}
In many cases, the above results can also be extended to Hermitian random matrices whose off-diagonal entries are complex-valued random variables. Although, the focus of this work has been the study of SISO systems, some generalizations to the case of MIMO systems is fairly straightforward. In addition, although our study is focused on controllability of random systems, due to the duality of controllability and observability, all of our results can be formulated naturally for observability and controllability of random systems. 

To the best of the authors' knowledge, this work is the first study of controllability of random systems and many questions are left for future works. One of the main open problems is to understand the extent to which the universality of controllability and minimal controllability holds. For example, the independence assumption on the entries of the Wigner matrices discussed in this work is essential to our analysis. This is certainly not the case when one considers Laplacian dynamics. The existence of such universality results for other ensembles of random matrices is left open for future studies. Also, controllability of a random matrix with an arbitrary deterministic input vector other than $e_i$ (such as the case of Conjecture~\ref{conj:conjecture2}) remains unanswered.

\bibliographystyle{plain}
\bibliography{control}
\end{document}